\documentclass[11pt,letterpaper,reqno]{amsart}
\usepackage{amsmath,amssymb,amsthm}
\newcommand{\C}{{\mathbb C}}
\newcommand{\F}{{\mathbb F}}
\newcommand{\Z}{{\mathbb Z}}
\newcommand{\conj}[1]{{\overline{#1}}}
\newcommand{\ceil}[1]{\lceil{#1}\rceil}
\newcommand{\sums}[1]{\sum_{\substack{#1}}}
\newcommand{\e}{\epsilon}
\newcommand{\leg}[2]{\mbox{$({#1}\!\mid\!{#2})$}}
\newcommand{\h}[1]{\hspace{#1}}
\newtheorem{theorem}{Theorem}[section]
\newtheorem{lemma}[theorem]{Lemma}
\newtheorem{corollary}[theorem]{Corollary}

\title{Littlewood Polynomials with Small $L^4$ Norm}
\author{Jonathan Jedwab, Daniel J. Katz, and Kai-Uwe Schmidt}
\date{17 June 2011 (revised 25 April 2013)}
\begin{document}

\thanks{J.~Jedwab and D.J.~Katz are with Department of Mathematics, Simon Fraser University, 8888 University Drive, Burnaby BC V5A 1S6, Canada. K.-U.~Schmidt was with Department of Mathematics, Simon Fraser University and is now with Faculty of Mathematics, Otto-von-Guericke University, Universit\"atsplatz~2, 39106 Magdeburg, Germany.
Email: {\tt jed@sfu.ca}, {\tt daniel\_katz\_2@sfu.ca}, {\tt kaiuwe.schmidt@ovgu.de}.}

\thanks{J.~Jedwab is supported by NSERC}
\thanks{K.-U. Schmidt was supported by German Research Foundation.}

\maketitle

\section*{Abstract}
Littlewood asked how small the ratio $||f||_4/||f||_2$ 
(where $||\cdot||_\alpha$ denotes the $L^\alpha$ norm on the unit circle) 
can be for polynomials~$f$ having all coefficients in $\{1,-1\}$,
as the degree tends to infinity.
Since 1988, the least known asymptotic value of this ratio has been 
$\sqrt[4]{7/6}$, which was conjectured to be minimum.
We disprove this conjecture by showing that there is a sequence of 
such polynomials, derived from the Fekete polynomials, 
for which the limit of this ratio is less than $\sqrt[4]{22/19}$.

\section{Introduction}
The $L^\alpha$ norm on the unit circle of polynomials having all coefficients 
in $\{1,-1\}$ (\emph{Littlewood polynomials}) has attracted sustained 
interest over the last sixty years \cite{Shapiro}, \cite{Erdos-57}, \cite{Rudin}, \cite{Newman}, \cite{Littlewood-66}, \cite{Littlewood-68}, \cite{Sarwate}, \cite{Newman-Byrnes}, \cite{Beck}, \cite{Borwein-Lockhart}.
For $1 \leq \alpha < \infty$, the $L^\alpha$ norm of a polynomial $f \in \C[z]$ on the unit circle is 
\[
||f||_\alpha=\left(\frac{1}{2\pi} \int_0^{2 \pi} |f(e^{i\theta})|^\alpha \,d\theta\right)^{1/\alpha},
\]
while $||f||_\infty$ is the supremum of $|f(z)|$ on the unit circle.
The norms $L^1$, $L^2$, $L^4$, and $L^\infty$ are of particular interest 
in analysis.

Littlewood was interested in how closely the ratio $||f||_\infty/||f||_2$ can 
approach $1$ as $\deg(f)\to\infty$ for $f$ in the set of polynomials now 
named after him~\cite{Littlewood-66}. 
Note that if $f$ is a Littlewood polynomial, then $||f||_2^2$ is 
$\deg(f)+1$.
In view of the monotonicity of $L^\alpha$ norms,
Littlewood and subsequent researchers used $||f||_4/||f||_2$ 
as a lower bound for $||f||_\infty/||f||_2$. 
The $L^4$ norm is particularly suited to this purpose because 
it is easier to calculate than most other $L^\alpha$ norms.
The $L^4$ norm is also of interest in the theory of communications, 
because $||f||_4^4$ equals the sum of squares of the aperiodic 
autocorrelations of the sequence 
formed from the coefficients of~$f$ \cite[eqn.~(4.1)]{Hoholdt-Jensen}, \cite[p.~122]{Borwein};
in this context one considers the {\it merit factor} 
$||f||_2^4/(||f||_4^4-||f||_2^4)$.
We shall express merit factor results in terms of 
$||f||_4/||f||_2$.

If $||f||_4/||f||_2$ is bounded away from $1$, then so is $||f||_\infty/||f||_2$, which would prove a modification of a conjecture due to Erd\H{o}s  \cite{Erdos-62}, \cite[Problem~22]{Erdos-57} asserting that there is some $c > 0$ such that $||f||_\infty/||f||_2 \geq 1+c$ for all non-constant polynomials $f$ whose coefficients have absolute value~$1$. It is known from Kahane's work that there is no such~$c$~\cite{Kahane}, but the modified conjecture where $f$ is restricted to be a Littlewood polynomial remains resistant \cite{Newman-Byrnes}.

Littlewood regarded calculations carried out by
Swinnerton-Dyer as evidence that the ratio $||f||_4/||f||_2$ can be 
made arbitrarily close to~$1$ for Littlewood polynomials~\cite{Littlewood-66}. 
However, he could prove nothing stronger than that this ratio is asymptotically
$\sqrt[4]{4/3}$ for the Rudin-Shapiro polynomials
\cite[Chapter III, Problem 19]{Littlewood-68}.
H{\o}holdt and Jensen, building on studies due to Turyn 
and Golay~\cite{Golay-83}, proved in 1988 that
this ratio is asymptotically $\sqrt[4]{7/6}$ for a sequence of Littlewood
polynomials derived from the Fekete polynomials~\cite{Hoholdt-Jensen}.
H{\o}holdt and Jensen conjectured that no further reduction in the
asymptotic value of $||f||_4/||f||_2$ 
is possible for Littlewood polynomials. 
Although Golay conjectured, based on heuristic reasoning, that the
minimum asymptotic ratio $||f||_4/||f||_2$ for Littlewood polynomials
is approximately $\sqrt[4]{333/308}$ \cite{Golay-82}, he later
cautioned that ``the eventuality must be considered that no systematic synthesis will ever be found which will yield [a smaller asymptotic ratio than $\sqrt[4]{7/6}$]''~\cite{Golay-83}. 

Indeed, for more than twenty years $\sqrt[4]{7/6}$ has remained the smallest 
known asymptotic value of $||f||_4/||f||_2$ for a sequence of 
Littlewood polynomials~$f$.
We shall prove that this is not the minimum asymptotic value.

\begin{theorem}\label{Introductory-Theorem}
There is a sequence $h_1, h_2, \ldots$ of Littlewood polynomials with 
$\deg(h_n) \to \infty$ and $||h_n||_4/||h_n||_2 \to \sqrt[4]{c}$ as $n\to\infty$, where $c < 22/19$ is the smallest root of $27 x^3 -498 x^2 + 1164 x -722$.
\end{theorem}

To date, two principal methods have been used to calculate the $L^4$ norm of 
a sequence of polynomials~\cite{Hoholdt}. The first is direct calculation, 
in the case
that the polynomials are recursively defined~\cite{Littlewood-68}.
The second, introduced by H{\o}holdt and Jensen \cite{Hoholdt-Jensen}
and subsequently employed widely for its generality 
\cite{Jensen-Hoholdt}, \cite{Jensen-Jensen-Hoholdt}, \cite{Borwein-Choi-2000}, \cite{Borwein-Choi-2001}, \cite{Borwein}, \cite{Borwein-Choi-2002}, \cite{Schmidt-Jedwab-Parker}, \cite{Jedwab-Schmidt}, 
obtains $||f||_4$ from a Fourier interpolation of~$f$.
In this paper, we use a simpler method that also 
obtains the $L^4$ norm of truncations and periodic extensions of~$f$.
We apply this method to
Littlewood polynomials derived from the Fekete polynomials,
themselves the fascinating subject of many studies \cite{Fekete-Polya}, \cite{Polya}, \cite{Montgomery}, \cite{Baker-Montgomery}, \cite{Conrey-Granville-Poonen-Soundararajan}, \cite{Borwein-Choi-Yazdani}, \cite{Borwein-Choi-2002}.
The possibility that these derived polynomials could have
an asymptotic ratio $||f||_4/||f||_2$ smaller than $\sqrt[4]{7/6}$ 
was first recognized by 
Kirilusha and Narayanaswamy~\cite{Kirilusha-Narayanaswamy} in 1999.
Borwein, Choi, and Jedwab subsequently used extensive numerical data to 
conjecture conditions under which the value $\sqrt[4]c$ in 
Theorem~\ref{Introductory-Theorem} could be attained 
asymptotically (giving a corresponding
asymptotic merit factor $1/(c-1)>6.34$)~\cite{Borwein-Choi-Jedwab}, 
but until now no explanation has been given as to whether their
conjecture might be correct, nor if so why.

\section{The Asymptotic $L^4$ Norm of Generalized Fekete Polynomials}

Henceforth, let $p$ be an odd prime and let $r$ and $t$ be integers with
$t \ge 0$. 
The {\it Fekete polynomial} of degree $p-1$ is
$\sum_{j=0}^{p-1} \leg{j}{p} z^j$, where $\leg{\cdot}{p}$ 
is the Legendre symbol. 
We define the {\it generalized Fekete polynomial}
\[
f_p^{(r,t)}(z)= \sum_{j=0}^{t-1} \leg{j+r}{p} z^j. 
\]
The polynomial $f_p^{(r,t)}$ is formed from the Fekete polynomial
of degree $p-1$ by cyclically permuting the coefficients through $r$ positions, 
and then truncating when $t < p$ or periodically extending
when $t > p$. 
We wish to determine the asymptotic 
behavior of the $L^4$ norm of the generalized Fekete polynomials
for all $r,t$.

Since the Legendre symbol is a multiplicative character, we can use 
\[
||f||_4^4=\frac{1}{2\pi}\int_0^{2\pi} \big[f(e^{i\theta})\conj{f(e^{i\theta})}\,\big]^2 d\theta
\]
to obtain
\begin{equation}\label{fprt44}
||f_p^{(r,t)}||_4^4=\sums{0 \leq j_1,j_2,j_3,j_4 < t \\ j_1+j_2=j_3+j_4} \leg{(j_1+r)(j_2+r)(j_3+r)(j_4+r)}{p}.
\end{equation}
Until now, the asymptotic
evaluation of~\eqref{fprt44} has been considered intractable because,
when $t$ is not a multiple of~$p$, the expression~\eqref{fprt44} 
is an incomplete character sum whose indices are subject to an additional 
constraint. 
We shall overcome this apparent difficulty by using the Fourier expansion
of the multiplicative character $\leg{j}{p}$ in terms of additive characters
of~$\F_p$, with Gauss sums playing the part of Fourier coefficients. 
This expansion introduces complete character sums over~$\F_p$ which, 
once computed, allow an easy asymptotic evaluation of~\eqref{fprt44}.
This method is considerably simpler and more general than the Fourier
interpolation method of \cite{Hoholdt-Jensen}.

\begin{theorem}\label{Fekete-Calculation}
Let $r/p \to R < \infty$ and $t/p \to T < \infty$ as $p \to \infty$. Then
\[
\frac{||f_p^{(r,t)}||_4^4}{p^2} \to - \frac{4 T^3}{3} 
+ 2 \sum_{n \in \Z} \max(0,T-|n|)^2 
+ \sum_{n \in \Z} \max(0,T-\left|T+2R-n\right|)^2
\]
as $p \to \infty$.
\end{theorem}
\begin{proof}
Let $\e_j=e^{2 \pi i j/p}$ for $j \in \Z$.
Gauss \cite{Gauss}, \cite{Berndt-Evans} showed that
\[
\sum_{k\in\F_p} \e_j^k \leg{k}{p} = i^{(p-1)^2/4} \sqrt{p} \, \leg{j}{p}.
\]
Substitution in \eqref{fprt44} gives
\[
||f_p^{(r,t)}||_4^4=\frac{1}{p^2} \sums{0 \leq j_1,j_2,j_3,j_4 < t \\ j_1+j_2=j_3+j_4} \enspace \sum_{k_1,k_2,k_3,k_4\in\F_p} \e_{j_1+r}^{k_1} \e_{j_2+r}^{k_2} \e_{j_3+r}^{k_3} \e_{j_4+r}^{k_4} \leg{k_1 k_2 k_3 k_4}{p}.
\]
Re-index with $k_1=x$, $k_2=x-a$, $k_3=b-x$, $k_4=c-x$ to obtain
\begin{equation}\label{The-Big-Formula}
||f_p^{(r,t)}||_4^4=\frac{1}{p^2} \sums{0 \leq j_1,j_2,j_3,j_4 < t \\ j_1+j_2=j_3+j_4} \enspace \sum_{a,b,c\,\in\F_p} \e_{a}^{-(j_2+r)} \e_b^{j_3+r} \e_c^{j_4+r} L(a,b,c),
\end{equation}
where
\[
L(a,b,c) = \sum_{x \in \F_p} \leg{x(x-a)(x-b)(x-c)}{p}.
\]
A Weil-type bound on character sums \cite{Weil}, \cite[Lemma 9.25]{Montgomery-Vaughan}, shows that $|L(a,b,c)| \leq  3 \sqrt{p}$ when $x(x-a)(x-b)(x-c)$ is not a square in $\F_p[x]$.
This polynomial is a square in $\F_p[x]$ if and only if it either has
two distinct double roots, in which case $L(a,b,c)=p-2$, or else has 
a quadruple root, in which case $L(a,b,c)=p-1$.
We shall see that contributions from $L(a,b,c)$ much smaller than $p$ will not influence the asymptotic value of the $L^4$ norm. 
Accordingly, we write $L(a,b,c)=M(a,b,c)+N(a,b,c)$, with a main term 
\[
M(a,b,c) = \begin{cases}
p & \text{if $x(x-a)(x-b)(x-c)$ is a square in $\F_p[x]$,} \\
0 & \text{if $x(x-a)(x-b)(x-c)$ is not a square in $\F_p[x]$,}
\end{cases}
\]
and an error term $N(a,b,c)$ satisfying
\begin{equation}\label{N-bound}
|N(a,b,c)| \leq 3 \sqrt{p}.
\end{equation}
There are three ways of pairing the roots of $x(x-a)(x-b)(x-c)$: (i) $a=c$ and $b=0$, (ii) $b=a$ and $c=0$, or (iii) $c=b$ and $a=0$.
So $M(a,b,c)=p$ if at least one of these conditions is met, and vanishes otherwise.
The only triple $(a,b,c)$ that satisfies more than one of these conditions is $(0,0,0)$.
We now reorganize \eqref{The-Big-Formula} by writing $L(a,b,c)$ as $M(a,b,c)+N(a,b,c)$, and then break the sum involving $M(a,b,c)$ into four parts: three sums corresponding to the three pairings, and a fourth sum to correct for the triple counting of $(a,b,c)=(0,0,0)$.  We keep the sum over $N(a,b,c)$ entire, and thus have
\begin{equation}\label{Five-Way-Sum}
||f_p^{(r,t)}||_4^4 = A+B+C+D+E,
\end{equation}
where 
\begin{eqnarray*}
A & = & \frac{1}{p} \sums{0 \leq j_1,j_2,j_3,j_4 < t \\ j_1+j_2=j_3+j_4} \enspace \sum_{a\in\F_p} \e_{a}^{j_4-j_2}, \\
B & = & \frac{1}{p} \sums{0 \leq j_1,j_2,j_3,j_4 < t \\ j_1+j_2=j_3+j_4} \enspace \sum_{b\in\F_p} \e_{b}^{j_3-j_2}, \\
C & = & \frac{1}{p} \sums{0 \leq j_1,j_2,j_3,j_4 < t \\ j_1+j_2=j_3+j_4} \enspace \sum_{c\in\F_p} \e_{c}^{j_3+j_4+2 r}, \\
D & = & -\frac{2}{p} \sums{0 \leq j_1,j_2,j_3,j_4 < t \\ j_1+j_2=j_3+j_4} 1, \\
E & = & \frac{1}{p^2} \sum_{a,b,c\,\in\F_p} N(a,b,c) \e_{-a+b+c}^r \sums{0 \leq j_1,j_2,j_3,j_4 < t \\ j_1+j_2=j_3+j_4} \e_{a}^{-j_2} \e_b^{j_3} \e_c^{j_4}.
\end{eqnarray*}
Note that $A=B$, and that $A$ counts the quadruples $(j_1,j_2,j_3,j_4)$ of integers in $[0,t)$ with $j_1+j_2=j_3+j_4$ and $j_4 \equiv j_2 \pmod{p}$, or equivalently, with $j_4-j_2=n p$ and $j_3-j_1=-n p$ for some $n \in \Z$.
For each $n\in\Z$ there are $\max(0,t-|n|p)$ ways to obtain
$j_4-j_2=n p$ and the same number of ways to obtain $j_3-j_1=-n p$. Therefore
$A=B=\sum_{n \in \Z} \max(0,t-|n| p)^2$.
This is a locally finite sum of continuous functions, and since
$t/p \to T$ as $p\to \infty$ we have
\[
\frac{A}{p^2} + \frac{B}{p^2} \to 2\sum_{n \in \Z} \max(0,T-|n|)^2.
\]

The summation in $D$ counts the quadruples $(j_1,j_2,j_3,j_4)$ of integers in $[0,t)$ with $j_1+j_2=j_3+j_4$.
For each $n\in\Z$ there are $\max(0,t-|t-1-n|)$
ways to represent $n$ as $j_1+j_2$ with $j_1, j_2 \in [0,t)$, and so
\[
D=-\frac{2}{p}\sum_{n \in \Z}\max(0,t-|t-1-n|)^2.
\] 
Thus $D = -\frac{2t(2t^2+1)}{3p}$, and
since $t/p \to T$ as $p \to \infty$ we get
$D/p^2 \to -4 T^3/3$.

Similarly, $C$ counts the quadruples $(j_1,j_2,j_3,j_4)$ of integers in $[0,t)$ with $j_1+j_2=j_3+j_4 = -2 r +np$ for some $n\in\Z$.
Replacing $n$ by $-2r+np$ in the above argument for $D$, we find that
\[
\frac{C}{p^2} = \sum_{n \in \Z} \max\left(0,\frac{t}{p}-\left|\frac{t-1+2r}{p}-n\right|\right)^2.
\]
This is a locally finite sum of continuous functions 
$\psi_n(x,y)=\max(0,x-|y-n|)^2$ evaluated at $x=t/p$ and $y=(t-1+2 r)/p$.
Since $r/p \to R$ and $t/p \to T$ as $p \to \infty$, it follows that
\[
\frac{C}{p^2} \to \sum_{n \in \Z} \max(0,T-|T+2 R-n|)^2.
\]

By \eqref{Five-Way-Sum}, it remains to show that $|E|/p^2 \to 0$ 
as $p \to \infty$. Use \eqref{N-bound} to bound $|N(a,b,c)|$, and then
use Lemma \ref{Technical-Lemma} below to bound the resulting outer sum over
$a,b,c$ to give $|E|/p^2 \le 192 p^{-7/2} \max(p,t)^3 (1+\log p)^3$.
Since $t/p \to T < \infty$ as $p \to \infty$, we then obtain $E/p^2 \to 0$
as required.
\end{proof}
We now prove the technical result invoked in
the proof of Theorem \ref{Fekete-Calculation}.
\begin{lemma}\label{Technical-Lemma}
Let $n$ be a positive integer and $\e_j=e^{2\pi i j/n}$ for $j \in \Z$. Then
\[
\sum_{a,b,c \, \in \Z/n\Z} \Bigg|\sums{0 \leq j_1,j_2,j_3,j_4 < t \\ j_1+j_2=j_3+j_4}  \e_a^{-j_2} \e_b^{j_3} \e_c^{j_4}\Bigg| \leq 64 \max(n,t)^3 (1+\log n)^3.
\]
\end{lemma}
\begin{proof}
Let $G$ be the entire sum. Re-index with $k=-a$, $\ell=c-b$, $m=b$, to obtain
\[
G=\sum_{k,\ell,m \in \Z/n\Z} \Bigg|\sums{0 \leq j_1,j_2,j_3,j_4 < t \\ j_1+j_2=j_3+j_4}  \e_k^{j_2} \e_\ell^{j_4} \e_m^{j_3+j_4}\Bigg|.
\]
Re-index the inner sum with $h=j_3+j_4$, separating into ranges $h \le t-1$ and $h \ge t$, so that $G \leq H+J$, where
\begin{align*}
H &= \sum_{k,\ell,m \in \Z/n\Z} \Bigg|\sum_{h=0}^{t-1} \e_m^h \sum_{j_2,j_4=0}^{h} \e_k^{j_2}  \e_\ell^{j_4}\Bigg|, \\
J &= \sum_{k,\ell,m \in \Z/n\Z} \Bigg|\sum_{h=t}^{2t-2} \e_m^h \sum_{j_2,j_4 = h-(t-1)}^{t-1} \e_k^{j_2} \e_\ell^{j_4} \Bigg|.
\end{align*}
We shall show that $H \leq 32 \max(n,t)^3 (1+\log n)^3$, from which we can deduce the same bound on $J$ after re-indexing with $h'=2(t-1)-h$, $j'_2=j_2+h'-(t-1)$, $j'_4=j_4+h'-(t-1)$, and $m^\prime=-(k+\ell+m)$.

Partition the sum $H$ into a sum with $k,\ell\not=0$, two sums where one of $k,\ell$ is zero and the other is nonzero, and a sum where $k=\ell=0$; then sum over the indices $j_2$ and $j_4$ to obtain $H=H_1+2H_2+H_3$, where
\begin{eqnarray}
H_1 & = & \sums{k,\ell,m \in \Z/n\Z \\ k,\ell\not=0} \left|\sum_{h=0}^{t-1} \frac{\e_m^h - \e_k \e_{m+k}^h -\e_\ell \e_{m+\ell}^h + \e_{k+\ell} \e_{m+k+\ell}^h}{(1-\e_k)(1-\e_\ell)} \right|,\nonumber\\
H_2 & = & \sums{k,m \in \Z/n\Z \\ k\not=0} \left|\sum_{h=0}^{t-1} \frac{(h+1) \left(\e_m^h -\e_k \e_{m+k}^h\right)}{1-\e_k}\right|,\label{H-parts} \\
H_3 & = & \sum_{m \in \Z/n\Z} \bigg|\sum_{h=0}^{t-1} (h+1)^2 \e_m^h \bigg|,\nonumber
\end{eqnarray}
To bound these sums, we prove by induction on $d \ge 0$ that for $s \ge 0$,
\begin{equation}\label{incomplete-inductive-bound}
\sums{j \in \Z/n\Z \\ j\not=0} \bigg|\sum_{h=0}^{s-1} (h+1)^d \e_j^h\bigg| \leq 2^{d+1} s^d n \log n.
\end{equation}
The base case follows from $\big|\sum_{h=0}^{s-1} \e_j^h\big| \leq 2/|1-\e_j|$ and the bound \cite[p.~136]{Davenport}
\begin{equation}\label{nlogn-Identity}
\sums{j=1}^{n-1} \frac{1}{|1-\e_j|} \leq n \log n.
\end{equation}
For the inductive step, apply the triangle inequality to the identity
$$\sum_{h=0}^{s-1} (h+1)^d \e_j^h = \sum_{g=0}^{s-1} \sum_{h=0}^{s-1} (h+1)^{d-1} \e_j^h - \sum_{g=0}^{s-1} \sum_{h=0}^{g-1} (h+1)^{d-1} \e_j^h,$$
to obtain a bound for the left hand side as the sum of the magnitudes of $2 s$ summations over $h$ involving $(h+1)^{d-1}\e_j^h$, then sum over
$j \ne 0$ and use the inductive hypothesis.

Now from \eqref{incomplete-inductive-bound} we find
\begin{equation}\label{inductive-bound}
\sums{j \in \Z/n\Z} \bigg|\sum_{h=0}^{t-1} (h+1)^d \e_j^h\bigg| \leq t^{d+1} + 2^{d+1} t^d  n \log n,
\end{equation}
since $\sum_{h=0}^{t-1} (h+1)^d \leq t^{d+1}$.
Apply the triangle inequality, \eqref{nlogn-Identity}, and \eqref{inductive-bound} to the expressions \eqref{H-parts} to obtain $H_1 \leq 4 (t+2 n\log n)(n \log n)^2$, $H_2 \leq 2 (t^2 + 4 t n \log n) n \log n$, and $H_3 \leq t^3 + 8 t^2 n \log n$. Therefore $H=H_1+2 H_2 + H_3 \leq 32 \max(n,t)^3 (1+\log n)^3$, as required.
\end{proof}

\section{Littlewood Polynomials with Small $L^4$ Norm}
The generalized Fekete polynomial $f_p^{(r,t)}$ is not
necessarily a Littlewood polynomial, because its coefficient of $z^j$ is $0$ 
for $0 \le j < t$ and $j+r \equiv 0 \pmod{p}$. 
Replace each such zero coefficient of $f_p^{(r,t)}$ with $1$
to define a family of Littlewood polynomials
\begin{equation}\label{g-Defn}
g_p^{(r,t)}(z) = f_p^{(r,t)}(z) \h{0.5em} + \h{-1.3em} 
\sums{0 \le j < t \\  j+r \equiv 0 \!\!\!\!\!\pmod{p}} \h{-1.3em} z^j.
\end{equation}
We now show that the asymptotic $L^4$ norm of 
$g_p^{(r,t)}$ as $p \to \infty$ behaves in the same way as that 
of $f_p^{(r,t)}$.

\begin{corollary}\label{Littlewood-Corollary}
Let $r/p \to R < \infty$ and $t/p \to T < \infty$ as $p \to \infty$. Then
\[
\frac{||g_p^{(r,t)}||_4^4}{p^2} \to - \frac{4 T^3}{3} 
+ 2 \sum_{n \in \Z} \max(0,T-|n|)^2 
+ \sum_{n \in \Z} \max(0,T-\left|T+2R-n\right|)^2.
\]
as $p \to \infty$.
\end{corollary}
\begin{proof}
Write $f=f_p^{(r,t)}$ and $g=g_p^{(r,t)}$ and $v = \ceil{t/p}$.
Since the $L^4$ norm of each $z^j$ in~\eqref{g-Defn} is~$1$, 
the triangle inequality for the $L^4$ norm gives
\[
\frac{1}{p^2}\Big|||g||_4^4 - ||f||_4^4\Big| \leq 
\frac{1}{p^2}\left(4 v ||f||_4^3+6 v^2 ||f||_4^2+4 v^3 ||f||_4+v^4\right).
\]
The limit as $p \to \infty$ of the right hand side is 0, because
$||f||_4/\sqrt{p}$ has a finite limit by Theorem~\ref{Fekete-Calculation}
and because $v/\sqrt{p} \to 0$ follows from $t/p \to T < \infty$.
\end{proof}
The specialization of Corollary~\ref{Littlewood-Corollary} to $T=1$ and $|R| \leq 1/2$ recovers the result due to H{\o}holdt and Jensen \cite{Hoholdt-Jensen} that the asymptotic ratio $||g_p^{(r,p)}||_4/||g_p^{(r,p)}||_2$ is $\sqrt[4]{7/6+8(|R|-1/4)^2}$ (which achieves its minimum value of $\sqrt[4]{7/6}$ at $R=\pm 1/4$).
The specialization of Corollary~\ref{Littlewood-Corollary} to $T \in (0,1]$ proves the conjecture of Borwein, Choi, and Jedwab \cite[Conjecture~7.5]{Borwein-Choi-Jedwab} mentioned at the end of the Introduction.
The authors of~\cite{Borwein-Choi-Jedwab} gave a proof that, subject to the
truth of their conjecture, Theorem \ref{Introductory-Theorem} holds.
In fact, Theorem~\ref{Introductory-Theorem} follows from
Corollary \ref{Littlewood-Corollary} directly.
We now show this, and demonstrate that the 
asymptotic ratio $||g_p^{(r,t)}||_4/||g_p^{(r,t)}||_2$ 
for $T>0$ and arbitrary $R$ cannot be made less than the value $\sqrt[4]{c}$
given in Theorem~\ref{Introductory-Theorem}.

\begin{corollary}\label{Global-Minimum-Corollary}
If $r/p \to R < \infty$ and $t/p \to T \in (0, \infty)$ as $p \to \infty$
then
$$\lim_{p \to \infty} \frac{||g_p^{(r,t)}||_4}{||g_p^{(r,t)}||_2} \geq \sqrt[4]{c},$$
where $c<22/19$ is the smallest root of $27x^3-498x^2+1164x-722$, with equality if and only if $T$ is the middle root $T_0$ of $4x^3-30x+27$ and $R=\frac14(3-2 T_0)+\frac{n}{2}$ for some integer~$n$.  
If $t/p \to \infty$ as $p \to \infty$, then $||g_p^{(r,t)}||_4/||g_p^{(r,t)}||_2 \to \infty$ as $p \to \infty$.
\end{corollary}

\begin{proof}
Recall that $||f||_2^2=t$ for a Littlewood polynomial~$f$ of degree~$t-1$.
In the case $t/p \to \infty$, the required result is an easy consequence of Lemma~\ref{L4-bound} below.
This leaves the case where $r/p \to R < \infty$ and 
$t/p \to T \in (0,\infty)$ as $p \to \infty$.
We have already noted that when $R=1/4$ and $T=1$, the asymptotic ratio 
$||g_p^{(r,t)}||^4_4/||g_p^{(r,t)}||^4_2$ is~$7/6$.
If $t/p > 3/2$, we know from Lemma~\ref{L4-bound} that $||g_p^{(r,t)}||^4_4/||g_p^{(r,t)}||^4_2 \geq 1 + 2 (1-p/t)^2 > 11/9 > 7/6$, and so we may assume $T \leq 3/2$.

By Corollary \ref{Littlewood-Corollary}, $\lim_{p\to\infty}||g_p^{(r,t)}||_4^4/||g_p^{(r,t)}||_2^4$ is
\[
\frac{1}{T^2} \left[-\frac{4 T^3}{3} + 2 \sum_{n \in \Z} \max(0,T-|n|)^2 + \sum_{n \in \Z} \max(0,T-\left|T+2R-n\right|)^2\right].
\]
Call this function $u(R,T)$ and note that it is always at least $-\frac{4 T}{3} +2$, so that $u(R,T) > 4/3$ if $T < 1/2$.
By combination with the previous bound on~$T$, we may assume $T \in [1/2,\,3/2]$.
Furthermore, $u(R,T)$ does not change when $R$ is replaced by $R+1/2$, so it is sufficient to consider points $(R,T)$ in the set $D=[0,\,1/2] \times [1/2,\,3/2]$.
We cover $D$ with six compact sets:
\begin{eqnarray*}
D_1 & = & \{(R,T) \in D: T+2 R \leq 1\}, \\
D_2 & = & \{(R,T) \in D: 1 \leq T + 2 R, \h{0.3em} T + R \leq 1\}, \\
D_3 & = & \{(R,T) \in D: 1 \leq T + R, \h{0.3em} T \leq 1 \}, \\
D_4 & = & \{(R,T) \in D: 1 \leq T, \h{0.3em}  T + R \leq 3/2\}, \\
D_5 & = & \{(R,T) \in D: 3/2 \leq T + R,  \h{0.3em} T + 2 R \leq 2\}, \\
D_6 & = & \{(R,T) \in D: 2 \leq T + 2 R \}.
\end{eqnarray*}
These sets are chosen so that the restriction of $u(R,T)$ to $D_k$ is a 
continuous rational function $u_k(R,T)$, and so $u(R,T)$ attains a 
minimum value on each~$D_k$. For example,
\[
u_4(R,T) = -\frac{4 T}{3} + 2 + 4 \frac{(T-1)^2}{T^2} + \frac{(1-2 R)^2}{T^2} + \frac{(2 T + 2 R -2)^2}{T^2}.
\]
For each $T$, the function $u_4(R,T)$ is minimized when $R=(3-2 T)/4$,
and $u_4(\tfrac14(3-2T),T)=\frac{1}{6T^2}(-8 T^3+48 T^2-60 T+27)$ is 
minimized on~$D_4$ when $T$ is the 
middle root~$T_0$ of $4 x^3-30 x+27$.
Let $R_0=(3-2 T_0)/4$. 
The point $(R_0,T_0)$ lies in the interior of $D_4$.
One can show that $u_4(R_0,T_0)$ is the smallest root $c$ of $27x^3-498x^2+1164x-722$, and that $c < 22/19$.

Following the same method, the minimum value of $u_3(R,T)$ on $D_3$ is~$7/6$, 
attained at $(1/4,1)$.
Partial differentiation with respect to~$R$ shows that the minimum 
of $u_2(R,T)$ on $D_2$ lies on the boundary with~$D_3$, and that the 
minimum of $u_5(R,T)$ on $D_5$ lies on the boundary with~$D_4$.
The involution $(R,T) \mapsto (1-R-T,T)$ maps $D_1$ onto $D_2$ while preserving the value of $u(R,T)$; 
likewise with $(R,T) \mapsto (2-R-T,T)$ for $D_6$ and $D_5$.
Therefore the unique global minimum of $u(R,T)$ on $D$ is $c$,
attained at $(R_0,T_0)$.
\end{proof}
We close by proving the bound on $||f||_4^4$ used in the proof of 
Corollary~\ref{Global-Minimum-Corollary}.
\begin{lemma}\label{L4-bound}
Let $m$ be a positive integer, and let $f(z)=\sum_{j=0}^{t-1} f_j z^j$ be a Littlewood polynomial for which $f_j=f_k$ whenever $j\equiv k \pmod{m}$.  Then 
\[
||f||_4^4 \geq \sum_{n \in \Z} \max(0,t-|n|m)^2.
\]
\end{lemma}
\begin{proof}
Note that $\conj{f(z)}=f(z^{-1})$ for $z$ on the unit circle.
By treating $f(z)$ and $f(z^{-1})$ as formal elements of $\C[z,z^{-1}]$, it is straightforward to show that $||f||_4^4$ is the sum of the squares of the coefficients of $f(z) f(z^{-1})$.
For each $n \in \Z$, the coefficient of $z^{n m}$ in $f(z) f(z^{-1})$ is $$\sums{0 \leq j,k < t \\ j-k = n m} f_j f_k.$$ 
By the periodicity of the coefficients of $f$, this equals the number of pairs of integers $(j,k)$ in $[0,t)$ with $j-k=n m$, which is $\max(0,t-|n| m)$.
Sum the square of this over $n \in \Z$ to obtain the desired bound.
\end{proof}


\end{document}